\documentclass[11pt, a4paper]{article}
\usepackage{authblk}
\usepackage[utf8]{inputenc}
\usepackage[T1]{fontenc}
\usepackage{natbib}
\usepackage[margin=20mm]{geometry}

\usepackage{amsmath,amsthm,amssymb}
\usepackage{bbm}

\usepackage{graphicx}
\usepackage{subcaption}
\usepackage{caption}
\graphicspath{{./figures/}}

\theoremstyle{definition}

\newtheorem{lemma}{Lemma}

\newtheorem{proposition}{Proposition}
\newtheorem{theorem}{Theorem}
\theoremstyle{remark}
\newtheorem{remark}{Remark}
\newtheorem{example}{Example}
\renewcommand{\Re}{\mathbb{R}}
\newcommand{\prob}{\mathbb{P}}
\newcommand{\expe}{\mathbb{E}}
\newcommand{\supp}{\mathrm{supp}}

\newcommand{\MEP}{\mathit{MEP}}
\newcommand{\IEP}{\mathit{IEP}}
\newcommand{\GG}{\mathit{GG}}
\newcommand{\ST}{\mathit{ST}}

\providecommand{\keywords}[1]
{
  \small	
  \textbf{\textit{Keywords---}} #1
}

\title{Entropy-based test for generalized Gaussian distributions}

\author[1]{Mehmet~Siddik~Cadirci}
\author[1]{Dafydd~Evans}
\author[1]{Nikolai~Leonenko}
\author[2]{Vitalii~Makogin \thanks{Corresponding author (e-mail: vitalii.makogin@uni-ulm.de).  The research was partially supported by DFG Grant 390879134.}}
\affil[1]{School of Mathematics, Cardiff University, Senghennydd~Road,~Cardiff,~Wales,~UK,~CF24~4AG.}
\affil[2]{Institute of  Stochastics, Ulm University, Ulm, 08069 Germany.}

\date{ \today}

\begin{document}
\markboth{IEEE TRANSACTIONS ON INFORMATION THEORY,~Vol.~00, No.~0, August~2020}%
{Cadirci \MakeLowercase{\textit{et al.}}: Entropy-based test for generalized Gaussian distributions}
%



\maketitle

\begin{abstract}
In this paper, we provide the proof of $L^2$ consistency for the $k$th nearest neighbour distance estimator of the Shannon entropy for an arbitrary fixed $k\geq 1.$ We construct the non-parametric test of goodness-of-fit for a class of introduced generalized multivariate Gaussian distributions based on a maximum entropy principle. The theoretical results are followed by numerical studies on simulated samples.
\end{abstract}


%

\section{Introduction}
%
%
%
%
We propose a non-parametric test of goodness-of-fit for a class of generalized multivariate Gaussian distributions. Our approach is based on the estimation of the differential (Shannon) entropy
\begin{equation}\label{eq:shannon}
H(f)= - \int_{\Re^m}f(x)\log f(x)\text dx 
\end{equation}

We use entropy estimators based on nearest neighbour distances. These were first studied by \citep{kozachenko1987} and subsequently by \citep{tsybakov,evans2002,goria2005, leonenko2008,leonenko2010,evans2008,penrose2011,delattre2017,gao2018,bulinski2019KL,bulinski2019shannon} and \cite{berrett2019a}. Nearest neighbour estimators (NNE) are particularly attractive because they are computationally efficient and generalise easily to the multivariate case. For an overview of non-parametric techniques of entropy estimation see \cite{beirlant1997}.

\citet{berrett2019a} have shown that subject to certain regularity conditions, as $k\to\infty$ the $k$-NNE of Shannon entropy is efficient only for $m\leq3$, and present a bias-corrected estimator for dimensions $m\geq4$. In this paper, we focus on the conventional $k$-NNE with fixed $k\geq1$, for which the asymptotic variance decrease rapidly up to $k=3$ only, see  \cite[Table 1]{berrett2019a}. It is interesting to note that this asymptotic inflation is distribution-free, which leads to the conjecture that $k=3$ is the most interesting case for any $m\geq 1$.

Entropy-based tests of goodness-of-fit exploit the so-called maximum entropy  principle \cite{vasicek1976, kapur1989}. \cite{choi2008} introduced  an entropy-based normality based on the fact that normal densities posses the largest Shannon entropy among all densities with the same variance, see also \cite{dudewicz1981,goria2005,evans2008} and the references therein. 
This paper proposes a new entropy-based test of generalized normality based on the maximum Shannon entropy principle for the generalized multivariate Gaussian distribution. 

The paper is organised as follows: we introduce the multivariate generalized Gaussian distribution in section~\ref{sec:GG} followed by the $k$-NNE of Shannon entropy in section~\ref{sec:entropyest}. In section~\ref{sec:maxent} we establish a maximum entropy principle for the generalized Gaussian distribution, and in section~\ref{sec:teststatistic} we present the associated goodness-of-fit statistics. Numerical results are included in section~\ref{sec:numerical}, with some auxiliary material deferred to Appendix~\ref{appx:lowerbound}. 

\section{Entropy estimation}\label{sec:entropyest}
Let $k\geq 1$ and $N>k$, and let  $\mathcal{X}_N=\{X_1, \ldots, X_N\}$ be a set of independent and identically distributed random vectors in $\Re^m$ with common density function $f$. Let $F$ be a finite subset of $\mathcal{X}_N$ having cardinality at least $k$, and let $\rho_k(x,F)$ denote the Euclidean distance between a point $x$ and its $k$th nearest neighbour in the set $F\setminus \{x\}$. 
The \emph{$k$th nearest neighbour estimator} ($k$-NNE) of the Shannon entropy $H(f)$ is defined to be
\begin{align}\label{eq:kNNE}
\widehat{H}_{N,k} 
	& = \frac{1}{N}\sum_{i=1}^{N}\log \left[\rho_k^m(X_i,\mathcal{X}_N) V_m  (N-1){e}^{-\psi(k)}\right], 
\end{align}
where $\psi(x)=\Gamma'(x)/\Gamma(x)$ is the digamma function and $V_m=\pi^{m/2}/\Gamma(m/2+1)$ is the volume of the unit ball in $\Re^m$.
For $k=1$, this reduces to 
\begin{equation}\label{eq:KL}
\widehat{H}_{N,1} = \frac{m}{N}\sum_{i=1}^N \log \rho_1(X_i,\mathcal{X}_N) + \log V_m + \gamma + \log (N-1).
\end{equation}
where $\gamma=-\psi(1)\approx 0.577216$ is the Euler-Mascheroni constant.
The estimator \eqref{eq:KL} was introduced by  \citet{kozachenko1987} while the general estimator \eqref{eq:kNNE} was first considered by \citet{goria2005}. The main properties of \eqref{eq:kNNE} have been studied by \cite{leonenko2008,leonenko2010,penrose2011,delattre2017,gao2018,bulinski2019KL,bulinski2019shannon,berrett2019a} and \cite{berrett2019b}.

Convergence in mean-square for the case $k=1$ was proved by \cite[Theorem 2.1.ii]{penrose2013}. 
\begin{theorem}[\cite{penrose2013}]
\label{thm:penrose}
Suppose that $\expe(\|X\|^{\alpha}) < \infty$ for some $\alpha>0$ and $f(x)\leq M$ for some $M>0$. Then
\[
\expe\left[\widehat{H}_{N,1}-H(f)\right]^2 \rightarrow 0 \quad \text{as $N\to\infty$.}
\]
 \end{theorem}
\begin{remark}
The condition of boundedness for the density $f$ is not explicitly stated by \cite[Theorem 2.4.ii]{penrose2013}. In Appendix~\ref{appx:lowerbound} we give an example of a density $f$ with bounded support and for which $H(f)$ is unbounded.
 \end{remark}
 
In this paper we prove the analogous result for arbitrary $k\geq 1$.
To this end we write \eqref{eq:kNNE} as
\[
\widehat{H}_{N,k} = \frac{1}{N}\sum_{x\in \mathcal{X}_N}l\left(N^{\frac{1}{m}} x,N^{\frac{1}{m}} \mathcal{X}_N\right),\]
where
\[
l(x,\mathcal{X}):=\log\left(\rho_k^m(x,\mathcal{X})V_me^{-\psi(k)}\right), x\in \Re^m.
\]
First, we require the following corollary of \cite[Theorem 3.1]{penrose2013}.
\begin{theorem}
\label{thm:penrose_specialcase}
Let $k\geq 1$ and $q=1$ or $q=2$, and suppose there exists $p\geq q$ such that 
\begin{equation}
\label{eq:cond-upperbound}
\sup_{N\geq k} \expe\left|l\left(N^{\frac{1}{m}} X_1,N^{\frac{1}{m}} \mathcal{X}_N\right)\right|^p<\infty.
\end{equation} 
Then we have $L^q$ convergence,
\begin{equation}
\nonumber
\frac{1}{N}\sum_{x\in \mathcal{X}_N}l\left(N^{\frac{1}{m}} x,N^{\frac{1}{m}} \mathcal{X}_N\right) 
    \to \int_{\Re^m} \expe\left[l(0,\mathcal{P}_{f(x)})\right]f(x)\,dx,
\end{equation}
as $N\to\infty$, where $\mathcal{P}_{\lambda}$ denotes a homogeneous Poisson point process of intensity $\lambda>0$ on $\Re^m$.
\end{theorem}

\begin{theorem}[Main theorem]
\label{thm:main}
Suppose that $\expe\|X\|^{\alpha}<\infty$ for some $\alpha>0$ and $f(x)\leq M$ for some $M>0$. Then for any fixed $k\geq 1$,
\begin{equation}
\expe\left[\widehat{H}_{N,k}-H(f)\right]^2 \to 0
\quad\text{as $N\to \infty$.}
\end{equation}
\end{theorem} 

\begin{proof}
We apply Theorem~\ref{thm:penrose_specialcase}. 
First, we show that
$
H(f) = \displaystyle\int_{\Re^m}\expe\big[l(0,\mathcal{P}_{f(x)})\big]f(x)\,dx,
$
where
\[
l\left(0,\mathcal{P}_{\lambda}\right)
	= m \log \rho_k(0,\mathcal{P}_\lambda) + \log V_m -\psi(k).
\]
Denote by $B_t(0)$ the (Euclidean) ball of radius $t$ centred at $0$ i.e, $B_t(0)=\{y\in \Re^m, \|y\|\leq t\}.$ 
The random variable  $\rho_k(0,\mathcal{P}_\lambda)$ is the distance from $0$ to the $k$th point of $\mathcal{P}_\lambda$, and thus has Erlang distribution with parameters $k$ and $\lambda|B_t(0)| = \lambda V_mt^m$, that is
\begin{align*}
&\prob\big(\rho_k(0,\mathcal{P}_\lambda)\leq t\big) 
	 = \prob\big(\mathcal{P}_\lambda\cap B_{t}(0) \geq k\big) \\
	& = 1 - \sum_{j=0}^{k-1}\frac{1}{j!}\big(\lambda|B_{t}(0)|\big)^j e^{-\lambda|B_{t}(0)|} \\
	& = 1 - \sum_{j=0}^{k-1}\frac{1}{j!}(\lambda V_mt^m)^j e^{-\lambda V_mt^m}
	\qquad (t\geq 0).
\end{align*} 
 Then
\begin{align*}
    &m\expe[\log \rho_k(0,\mathcal{P}_\lambda)]\\
    &=\int_{0}^\infty \log t^m \frac{(\lambda V_m)^k  (t^m)^{(k-1)}}{(k-1)!}e^{-\lambda V_m t^m} m t^{m-1}dt\\
    &=-\log (\lambda V_m) + \int_{0}^\infty \log y \frac{y^{k-1}}{(k-1)!}e^{-y}dy\\
    &=-\log \lambda  -\log V_m  + \psi(k).
\end{align*}

Hence $\expe\big[l(0,\mathcal{P}_{\lambda})\big] = -\log\lambda$ and thus $H(f)$ equals
\[
 -\int_{\Re^m} f(x)\log f(x)\,dx 
	 = \int_{\Re^m}\expe\big[l(0,\mathcal{P}_{f(x)})\big]f(x)\,dx.
\]

Second, we check condition \eqref{eq:cond-upperbound}. 
Note that for every $\delta\in (0,1)$ and $p>1$ there exists $C>0$ such that 
\[
|\log t|^p\leq Ct^{-\delta}\mathbbm{1}_{[0,1]}(t)+C t^{\delta}\mathbbm{1}_{[1,\infty)}(t),\qquad t>0.
\]
Then because
\begin{align*}
|l(x,\mathcal{X})|^p
	& = \big|\log V_m - \psi(k) + \log\rho_k^m(x,\mathcal{X})\big|^p \\
	& \leq \big|\log V_m - \psi(k)\big|^p  
		+ \big|\log\rho_k^m(x,\mathcal{X})\big|^p 
\end{align*}
we have
\begin{align}
\nonumber&\frac{1}{2^{p-1}}\expe\left|l\left(N^{\frac{1}{m}} X_1,N^{\frac{1}{m}} \mathcal{X}_N\right)\right|^p \\
	& \leq \left|\log V_m - \psi(k)\right|^p+ \expe\left|\log \rho_k^m\left(N^{\frac{1}{m}} X_1,N^{\frac{1}{m}} \mathcal{X}_N\right) \right|^p \nonumber \\
	&\leq \left|\log V_m - \psi(k)\right|^p \nonumber \\
	& +  C\expe \rho_k^{-\delta}\left(N^{\frac{1}{m}} X_1,N^{\frac{1}{m}} \mathcal{X}_N\right) 
		\mathbbm{1}_{[0,1]}\left[\rho_k^{\delta}\left(N^{\frac{1}{m}} X_1,N^{\frac{1}{m}} \mathcal{X}_N\right)\right] 
	\label{thm2:eq3} \\
	& +  C\expe \rho_k^{\delta}\left(N^{\frac{1}{m}} X_1,N^{\frac{1}{m}} \mathcal{X}_N\right) 
		\mathbbm{1}_{[1,\infty)}\left[\rho_k^{\delta}\left(N^{\frac{1}{m}} X_1,N^{\frac{1}{m}} \mathcal{X}_N\right)\right].
	\label{thm2:eq4}    
\end{align}

Term \eqref{thm2:eq3} is finite because
\begin{align}
    \nonumber&\sup_{N\geq k} \expe \rho_k^{-\delta}\left(N^{\frac{1}{m}} X_1,N^{\frac{1}{m}} \mathcal{X}_N\right) \mathbbm{1}_{[0,1]}\left(\rho_k^{\delta}\left(N^{\frac{1}{m}} X_1,N^{\frac{1}{m}} \mathcal{X}_N\right)\right) \\
    \label{cond2}&\leq \sup_{N\geq k} \expe \rho_1^{-\delta}\left(N^{\frac{1}{m}} X_1,N^{\frac{1}{m}} \mathcal{X}_N\right) <\infty,
\end{align}
where \eqref{cond2} is ensured  by  \cite[Lemma 7.5]{penrose2013} since $f$ is bounded and $\delta\in (0,m).$ 

Let $r_c(f):=\sup\{r\geq 0: \expe\|X_1\|^r <\infty\}$.
In the proof of \cite[Theorem 2.3]{penrose2011} we see that
if $r_c(f)>0$ and $0<\delta<mr_c(f)(m+r_c(f))^{-1}$, then 
\[
\sup_{N\geq k} \expe\rho_k^{\delta}\left(N^{\frac{1}{m}} X_1,N^{\frac{1}{m}} \mathcal{X}_N\right)<\infty.
\] 

Thus, term  \eqref{thm2:eq4} is finite.
\end{proof}

\section{The generalized Gaussian distribution}\label{sec:GG}
The \emph{multivariate exponential power distribution} $\MEP_m(s,\mu,\Sigma)$ on $\Re^m$ 
\cite{solaro2004}
has the density function
\begin{equation}\label{eq:mep}
f(x\,;m,s,\mu,\Sigma) 
	= \frac{\Gamma(m/2+1)}{\pi^{m/2}\Gamma(m/s+1)2^{m/s}\sqrt{\det \Sigma}}
	\exp\left(-\frac{1}{2}\Big[(x-\mu)^{\rm T}\Sigma^{-1}(x-\mu)\Big]^{s/2}\right)
\end{equation}
where $\mu\in\Re^m$ is the mean vector, $\Sigma$ is an $m\times m$ positive definite matrix, $s>0$ is a shape parameter \cite{solaro2004}, and variance-covariance matrix $V =\beta\Sigma$ where 
\begin{equation}\label{eq:mep-scalefactor}
\beta(m,s) = \frac{2^{2/s}\Gamma\big[(m+2)/s\big]}{m\Gamma(m/s)}.
\end{equation}
Note that $s=2$ corresponds to the multivariate normal distribution $N(\mu,\Sigma)$ on $\Re^m$, while $s=1$ corresponds to the multivariate Laplace distribution. 
Taking $\mu$ to be the null vector and $\Sigma$ to be the identity matrix, we obtain the  \emph{isotropic exponential power distribution} $\IEP_m(s)$ on $\Re^m$, 
\begin{equation}\label{eq:sep}
f(x\,;m,s) 
	= \frac{\Gamma(m/2+1)}{\Gamma(m/s+1)\pi^{m/2}2^{m/s}}
	\exp\left(-\frac{1}{2}\|x\|^s\right),
\end{equation}
$ x\in \Re^m,$ where $\|\cdot\|$ denotes the Euclidean norm on $\Re^m$. 

Applying the scaling $x\mapsto(2\tau)^{-1/s}x$ for $\tau>0$ yields the \emph{generalized Gaussian} distributions $\GG_\tau(m,s)$ on $\Re^m$, with density functions
\begin{equation}\label{eq:GGpdf}
f_c(x;m,s) = c(m,s)\exp\left(-\tau\|x\|^s\right), x\in \Re^m,\\
\end{equation}
where
\[
	\quad c(m,s) = \frac{\Gamma(m/2+1)\tau^{m/s}}{\Gamma(m/s+1)\pi^{m/2}},
\]
and taking $\tau=1/s$ yields the canonical distribution
\begin{equation}\label{eq:GGpdf-canonical}
f(x;m,s) = c_0(m,s)\exp\left(-\frac{\|x\|^s}{s}\right),  x\in \Re^m,
\end{equation}
where
\[
c_0(m,s) = \frac{\Gamma(m/2+1)}{\Gamma(m/s+1)\pi^{m/2}s^{m/s}}.
\]

\subsubsection*{Moments}
A random vector $X\in\Re^m$ is called \emph{isotropic} if its density $f$ can be written as $f(x) = \tilde{f}(\|x\|)$ for some function $\tilde{f}:\Re\to[0,\infty)$ called the \emph{radial density}, where $\|\cdot\|$ is the Euclidean norm on $\Re^m$.
If $X$ is isotropic and $g: \Re\rightarrow \Re$ is a Borel function, it is easy to show that
\begin{equation}
    \expe\left[g(\|X\|)\right] = \int_{\Re^m}g(\|x\|)f(x)\,dx = \frac{2\pi^{m/2}}{\Gamma(m/2)}\int_{0}^{\infty}g(r)\tilde{f}(r)r^{m-1}\,dr
\end{equation}
provided the integrals exist. 
In particular, the moments of order $s>0$ are given by
\begin{equation}\label{eq:s-moment}
\expe(\|X\|^s) 
    = \frac{2\pi^{m/2}}{\Gamma(m/2)}\int_{0}^{\infty}r^{m+s-1}\tilde{f}(r)\,dr 
\end{equation}
provided the integrals exist.

\begin{lemma}\label{lem:moments-GG}
If $X\sim \GG_\tau(m,s)$, then $\expe(\|X\|^s) = m/s\tau$.
\end{lemma}

\begin{proof}
If $X\sim \GG_\tau(m,s)$ then $X$ is isotropic and has radial density function
\[
\tilde{f}(r) = \frac{\Gamma(m/2+1)\tau^{m/s}}{\Gamma(m/s+1)\pi^{m/2}}\exp(-\tau r^s).
\]
Hence by \eqref{eq:s-moment} we have
\begin{align*}
\expe(\|X\|^s) 
	&= \frac{2\pi^{m/2}}{\Gamma(m/2)}\int_{0}^{\infty}r^{s}\tilde{f}(r)r^{m-1}\,dr\\
	& = \frac{m\tau^{m/s}}{\Gamma(m/s+1)}\int_{0}^{\infty}r^{m+s-1}\exp(-\tau r^s)\,dr
\end{align*}

and changing the variable of integration to $t=\tau r^s$ yields
\begin{align*}
\expe(\|X\|^s) 
	& = \frac{m}{s\tau \,\Gamma(m/s+1)}\int_{0}^{\infty}t^{m/s}e^{-t}dt
	= \frac{m}{s\tau }.
\end{align*}
\end{proof}

\section{A maximum entropy principle for $GG_\tau (m,s)$}\label{sec:maxent}
It is well known \cite{kapur1989} that among all distributions on $\Re^m$ whose densities $f$ are supported on the whole of $\Re^m$ and whose mean and covariance matrix are fixed at zero and $\Sigma$ respectively, the differential entropy $H(f)$ is maximised by the multivariate Gaussian distribution $N(0,\Sigma)$ on $\Re^m$, and thus
\begin{equation} 
\label{eq:maxent}
H(f)\leq \log\left[(2\pi e)^{m/2}\sqrt{\det\Sigma}\right].
\end{equation}

We now prove an analogous result for the generalized Gaussian distribution.

\begin{theorem}
\label{thm:maxentGG}
Let $X\in\Re^m$ be a random vector, whose density $f$ is supported on the whole of $\Re^m$, and for which there exists some $s>0$ such that $\expe(\|X\|^s) < \infty$. Then $H(f)$ is finite and satisfies
\[
H(f)\leq \frac{m}{s}\log\big(c_1(m,s)\expe\|X\|^s\big),
\]
where
\[c_1(m,s) = \left(\frac{\pi^{m/2}\Gamma(m/s+1)}{\Gamma(m/2+1)}\right)^{s/m}\left(\frac{se}{m}\right)
\]
with equality if and only if $X \sim\GG_\tau(m,s)$ with $\tau =m/(s\expe\|X\|^s)$.
\end{theorem}

\begin{proof}
Let $X$ and $Z$ be two random vectors whose density functions, $f$ and $f^*$ respectively, are supported on the whole of $\Re^m$, and for which there exists some $s>0$ with $\expe\|X\|^s = \expe\|Z\|^s < \infty$. First, we observe that
\begin{equation}\label{eq:shannon-upperbound}
H(f) 
	\leq - \int_{\Re^m}f(x)\log f^*(x)\,dx,
\end{equation}
with equality if and only if $f = f^*$ almost everywhere. This follows by Jensen's inequality, 
\begin{align*}
&-\int_{\Re^m}f(x)\log f(x)\,dx + \int_{\Re^m}f(x)\log f^*(x)dx\\ 
    & = \int_{\Re^m}f(x)\log \left(\frac{f^*(x)}{f(x)}\right)\,dx \\
	& \leq \log\left(\int_{\Re^m}f(x) \left(\frac{f^*(x)}{f(x)}\right)\,dx\right) \\
	& = \log \left(\int_{\Re^{m}} f^*(x)dx\right) 
	  = 0.
\end{align*}

If $Z\sim\GG_\tau(m,s)$ with $\tau =\displaystyle\frac{m}{s\expe\|X\|^s}$ (which ensures that $\expe\|Z\|^s=\expe\|X\|^s$) we have
\[
f^*(x) = c(m,s)\exp(-\tau\|x\|^s), 
\]
where
\[
c(m,s) = \frac{\Gamma(m/2+1)\tau^{m/s}}{\Gamma(m/s+1)\pi^{m/2}}.
\]

For this case, $-\log f^*(x) = \tau \|x\|^s - \log c(m,s)$ and hence
\begin{align*}
&-\int_{\Re^m}f(x)\log f^*(x)dz 
	\\
	& = \tau\int_{\Re^m}\|x\|^sf(x)\,dx - (\log c(m,s)) \int_{\Re^m}f(x)\,dx  \\[1ex]
	& = \tau\,\expe\|X\|^s - \log c(m,s) \\
	& = \frac{m}{s} - \log c(m,s) \quad\text{by Lemma~\ref{lem:moments-GG}.} 
\end{align*}

Thus by~\eqref{eq:shannon-upperbound} and substituting for $c(m,s)$ we obtain 
\begin{align*}
H(f) 
	& \leq \frac{m}{s} - \log\left[\frac{\tau^{m/s}\Gamma(m/2+1)}{\pi^{m/2}\Gamma(m/s+1)}\right]\\
	&= \frac{m}{s}\log \left[\left(\frac{\Gamma(m/s+1)}{\Gamma(m/2+1)}\right)^{\frac{s}{m}}\left(\frac{e\pi^{\frac{s}{2}}}{\tau}\right)\right],
\end{align*}
and substituting for $\expe\|X\|^s = m/(s\tau)$ completes the proof.
\end{proof}

\begin{remark}
Theorem $\ref{thm:maxentGG}$ was proved for $m=1$ in \cite{wyner1969} and \cite[p.103-104]{rosenblatt2000}. For $m\geq 1$, some statements of Theorem~\ref{thm:maxentGG} were also proved using other methods in \cite{lutwak2007}.
\end{remark}

\section{A test statistic for $GG(m,s)$}\label{sec:teststatistic}
Let $k\geq 1$ be fixed and $\mathcal{K}$ be the class of density functions $f$ on $\Re^m$ such that 
\begin{enumerate}
\item $\supp(f)=\Re^m$, 
\item $\expe(\|X\|^s)<\infty$ for some $s>0$, 
\item $\expe(\widehat{H}_{N,k}) \to H(f)$ as $N\to\infty$, and 
\item $\widehat{H}_{N,k} \to H(f)$ in probability as $N\to\infty$.
\end{enumerate}

\begin{proposition}
The $\GG_\tau(m,s)$ density functions belong to $\mathcal{K}$ for all $m\geq 1$, $s>0$, $c>0$ and $k\geq 1$.
\end{proposition}

\begin{proof}
The statement follows from Theorem~\ref{thm:main}, which applies because $f$ is bounded, and Lemma \ref{lem:moments-GG}.
\end{proof}

Let $X\in\Re^m$ be a random vector with density $f\in\mathcal{K}$, and let $s>0$ be fixed. Based on a random sample $X_1,X_2,\dots$ from the distribution of $X$, we use the maximum entropy principle proved in Section~\ref{sec:maxent} to test the hypothesis $X\sim\GG(m,s)$ against a suitable alternative.
By Theorem~\ref{thm:maxentGG}, if $X\sim\GG(m,s)$ then
\[
H(X) = \frac{m}{s}\log\expe\|X\|^s + \frac{m}{s} \log c_1(m,s),\]
where
\[
c_1(m,s)=\left(\frac{\pi^{m/2}\Gamma(m/s+1)}{\Gamma(m/2+1)}\right)^{s/m}\left(\frac{se}{m}\right).
\]
We estimate the entropy $H(X)$ by the $k$th nearest neighbour estimator 
\[
\widehat{H}_{N,k} 
	= \frac{m}{N}\sum_{i=1}^{N}\log\rho_k(X_i,\mathcal{X}_N) + \log V_m + \log(N-1) -\psi(k),
\]
and the moment $\expe\|X\|^s$ by the sample moment
\[
\bar{X}^{(s)}_{N} = \frac{1}{N}\sum_{i=1}^{N}\|X_i\|^s.
\]
Our test statistic $T_{N,k}=T_{N,k}(m,s)$ is then
\begin{align*}
 T_{N,k} 
	& = \widehat{H}_{N,k} - \frac{m}{s}\log\bar{X}^{(s)}_{N} - \frac{m}{s} \log c_1(m,s).
\end{align*}

By the law of a large numbers, $\bar{X}^{(s)}_N\to\expe\|X\|^s$ in probability as $N\to\infty$. Hence by Slutsky's theorem, if $X\sim\GG_m(s)$ then for any fixed $k\in\{1,\ldots,N-1\}$ we have
\[
T_{N,k} \to 0 \quad\text{in probability as $N\to\infty$.}
\]
Otherwise, by the maximum entropy principle it must be that $T_{N,k}\to \xi$ in probability as $N\to\infty$, where the constant $\xi = \xi(m,s,k)$ is strictly negative.
Thus we reject the hypothesis $X\sim\GG(m,s)$ whenever $T_{N,k}\geq t_{N,k,\alpha}$, where $t_{N,k,\alpha}=t_{N,k,\alpha}(m,s)$ a so-called \emph{critical value} of the test statistic $T_{N,k}(m,s)$ at significance level $\alpha$, which is a solution of
\[
\prob_{H_0}\big(T_{N,k}\geq t\big) = \alpha. 
\]
An analytical derivation of the distribution of $T_{N,k}$ when $X\sim\GG(m,s)$ is difficult because the covariances of $T_{N,k}$ and $\bar{X}^{(s)}_N$ are intractable, even though the \emph{asymptotic} behaviour of $\widehat{H}_{N,k}$ can be revealed by applying results of \cite{penrose2011, delattre2017} or \cite{berrett2019a}, and the asymptotic behaviour of $\bar{X}^{(s)}_N$ by the delta method. Thus we use Monte Carlo simulation to investigate the distribution of $T_{N,k}=T_{N,k}(m,s)$ for different combinations of parameter values.

\begin{remark}
The test statistic $T_{N,k}$ is scale-invariant: if $Y = aX$ for some $a>0$, then
$\widehat{H}_{N,k}(Y) = \log(a^m) + \widehat{H}_{N,k}(X)$ and $\bar{Y}^{(s)}_{N} = a^{s}\bar{X}^{(s)}_{N}$, and hence
\begin{align*}
T_{N,k}(Y) 
	 &= \widehat{H}_{N,k}(Y) - \frac{m}{s}\log\bar{Y}^{(s)}_{N} -  \frac{m}{s} \log c_1(m,s) \\
	& = \log(a^m) + \widehat{H}_{N,k}(X) - \frac{m}{s}\log(a^s) \\
	&- \frac{m}{s}\log\bar{X}^{(s)}_{N} - \frac{m}{s} \log c_1(m,s) \\
	 &= \widehat{H}_{N,k}(X) - \frac{m}{s}\log\bar{X}^{(s)}_{N} - \frac{m}{s} \log c_1(m,s)\\
	 &= T_{N,k}(X).
\end{align*}
\end{remark}

\section{Numerical results}\label{sec:numerical}

To investigate the behaviour of $T_{N,k}(m,s)$ we generate random samples from the $\GG(m,s)$ distribution, and also from the isotropic multivariate Student $t$-distribution $\ST(m,\nu),$ $\nu>0$ on $\Re^m$ which has density function
\[
f(x;m,\nu) 
	= \frac{\Gamma\big[(\nu + m)/2\big]}{\Gamma(\nu/2)(\nu\pi)^{m/2}}
	\left(1+\frac{\|x\|^2}{\nu}\right)^{-(\nu+m)/2}, x\in \Re^m.
\]
This is achieved via the following stochastic representations \cite{solaro2004}.

\begin{lemma}
\begin{enumerate}
\item 
For $X\sim\GG(m,s)$ we have $X\stackrel{d}{=}UR$ where $U$ is uniformly distributed on $\mathbb{S}^{m-1}$ and $R\stackrel{d}{=}V^{1/s}$ with $V\sim\text{Gamma}(m/s,2)$.
\item
For $X\sim\ST(m,\nu)$ we have $X\stackrel{d}{=}Z/\sqrt{G}$, where $Z\sim N(0,I_m)$ and $G\sim\text{Gamma}(\nu/2, 2\nu)$.
\end{enumerate}
\end{lemma}

For the case $m=2,$ we put the generated points on scatter plots for different values of $s$ and $\nu,$ see Figure \ref{fig:scatter-GG} for $\GG(m,s)$ and Figure \ref{fig:scatter-ST} for $\ST(m,\nu).$ One can observe that visually distributions $\GG(m,s)$ and $\ST(m,\nu)$ are hardly distinguishable. Therefore, we apply our goodness-of-fit test for detecting of the generalized Gaussian distribution.

\begin{figure*}[!t]
\centering
\includegraphics[width=\textwidth]{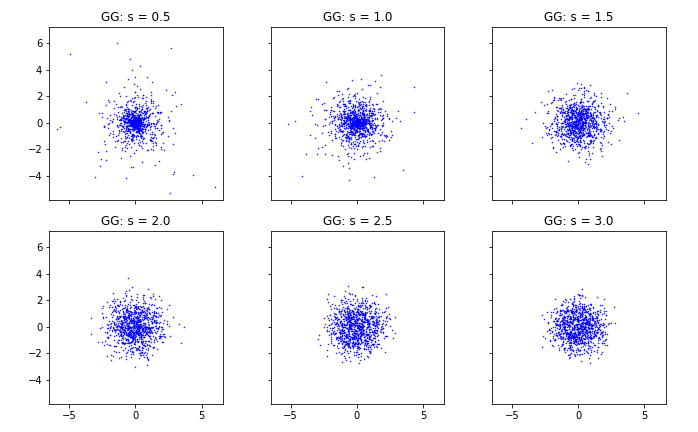}
\caption{Scatter plots for $\GG(m,s)$ with $m=2$}\label{fig:scatter-GG}
\end{figure*}

\begin{figure*}[!t]
\centering
\includegraphics[width=\textwidth]{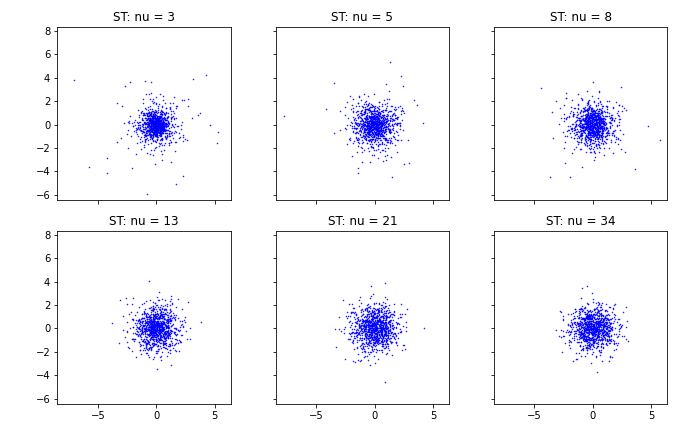}
\caption{Scatter plots for $\ST(m,\nu)$ with $m=2$}\label{fig:scatter-ST}
\end{figure*}

\subsection{Empirical distribution of $\GG(m,s)$}
We generate $N=10^6$ points from the $\GG(m,s)$ distribution for different values of $s$. For the purpose of comparison we apply the scaling $X\mapsto X/\sigma$ where
\[
\sigma^2 = \frac{2^{2/s}\Gamma\big[(m+2)/s\big]}{m\Gamma(m/s)}
\]
is the variance of the $\GG(m,s)$ distribution. The results are shown in Figure~\ref{fig:empirical}.

\begin{figure*}[!t]
\centering
\subfloat[Empirical density function]{\includegraphics[width=0.49\textwidth]{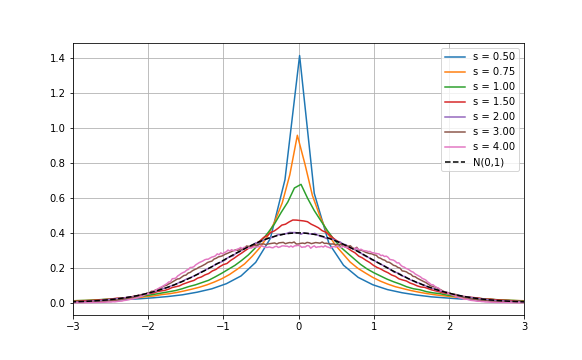}\label{fig:epdfs}}
\hfill
\subfloat[Empirical log--density function]{\includegraphics[width=0.49\textwidth]{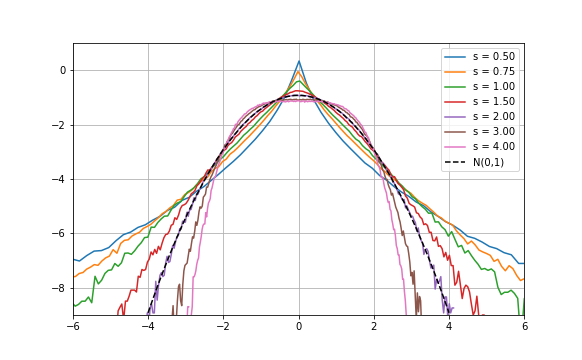}\label{fig:elogpdfs}}
\caption{Empirical distribution of $\GG(m,s)$ for $m=1$ and different values of $s$.}
\label{fig:empirical}
\end{figure*}

\subsection{Asymptotic behaviour of $T_{N,k}(m,s)$ as $N\to\infty$.}

For fixed $(N,k)$ and $(m,s)$ we generate a sample of size $N$ from the $\GG(m,s)$ distribution and record the empirical value of $T_{N,k}(m,s)$, repeating this $M=10$ times. This yields a sample realisation $\{T_1,T_2,\ldots,T_M\}$ from the distribution of $T_{N,k}(m,s)$, from which we estimate its mean and variance by
\[
\bar{T}_{N,k} = \frac{1}{M}\sum_{j=1}^M T_i
\quad\text{and}\quad
S^2_{N,k} = \frac{1}{M-1}\sum_{j=1}^M (T_i - \bar{T}_{N,k})^2
\]

In Figures~\ref{fig:consistency-kall} and~\ref{fig:consistency-sall} we show how $\bar{T}_{N,k}(m,s)$ approaches 0 as $N$ increases for various values of $m$, $s$ and $k$. In Figure~\ref{fig:consistency-k=1} we show its behaviour for $k=1$ with error bars corresponding to the standard error $S_{N,k}/\sqrt{M}$.

\subsubsection{Asymptotic behaviour of $T_{N,k}(m,s_0)$ on data from $GG(m,s_1)$}

For various values of $s_0$ and $s_1$ we generate samples from the $\GG(m,s_1)$ distribution and examine the behaviour of $T_{N,k}(m,s_0)$ as $N$ increases. The results are shown in Figure~\ref{fig:convergence-GG} and Figure~\ref{fig:convergence-GGall}. When $s_0\neq s_1$ we see that the statistic approaches a strictly negative value, and that this becomes increasingly negative as the difference between $s_0$ and $s_1$ increases.

\subsubsection{Asymptotic behaviour of $T_{N,k}(m,s)$ on data from $\ST(m,\nu)$}

For various values of $s$ we generate samples from the $\ST(m,\nu)$ distribution and examine the behaviour of $T_{N,k}(m,s)$ as $N$ increases. The results are shown in Figure~\ref{fig:convergence-STall}.

\subsection{Empirical distribution of $T_{N,k}(m,s)$}

Numerical results suggest that the distribution of $T_{N,k}(m,s)$ is asymptotically normal as the sample size $N\to\infty$. For different values of $(N,k)$ and $(m,s)$ we generate $N_T=1000$ samples from the $\GG(m,s)$ distribution and record the corresponding values of $T_{N,k}(m,s)$, repeating this $M=10$ times. To each of these $10$ samples from the distribution of $T_{N,k}(m,s)$ we then apply the Shapiro-Wilk test for normality \cite{shapiro1965} and record the $p$-value returned by the test.

Figure~\ref{fig:sw-pvals-all} shows how these $p$-values behave as $N$ increases, for various values of $m$, $s$ and $k$. The plots suggest that the normal hypothesis cannot be rejected for samples of size $N=200$ or more. 

In Figure~\ref{fig:sw-pvals-k=1} we show how the $p$-values behave for $k=1$, with error bars corresponding to the standard error across the $M=10$ repetitions.

\section*{Acknowledgment}
Nikolai Leonenko would like to thank Prof. Richard  Samworth and Prof. Mathew Penrose for a fruitful discussion on a problem of entropy estimation on the Workshop 'Estimation of entropies and other functionals: Statistics meets information theory' on 9-11 September 2019,	Cambridge (UK).

\bibliographystyle{biometrika}
\bibliography{entropy}

\appendix
\section{Lower bound on Shannon entropy}\label{appx:lowerbound}
Below we present some essentials about lower bounds of Shannon entropy. First, we show that there exist densities such that $-\infty = H(f)<\infty.$  We modify an example of \cite[p.223]{gnedenko1954}. For other examples, see \cite{barron1986}.
\begin{example}
Let $m=1$ and consider the density
\begin{equation}\label{eq_1_dim}
 f(x) = \left[ x\log^2 \frac{e}{x} \right]^{-1}\mathbbm{1}_{[0,1]}(x), \qquad x \in \Re.
\end{equation}
If $X$ is random variable with density $~\eqref{eq_1_dim}$, then for $s=1$
\begin{equation}\label{eq_1_E(x)}
\expe X = \expe|X| =   \int_{0}^{1}\left[\log^2\frac{e}{x}\right]^{-1}\text dx = 1- \text E_1(1) \simeq 0.40365...,
\end{equation}
where
\begin{equation*}
\text E_p(z)  = z^{p-1}\Gamma(1-p,z) = z^{p-1} \int_{z}^{\infty} \frac{e^{-zt}}{t^p} \text dt, \quad p>0,  \quad z\geq0,
\end{equation*}
is the generalized exponential integral. Thus by Theorem~\ref{thm:maxentGG}
with $m=1$ and $s=1$,
\[
H(f) \leq \log\left[2e\expe|X| \right] \simeq  0.8073.
\]
From the other hand,
\[
H(f)=-\int_{0}^{1}\left[x\log^2\frac{e}{x}\right]^{-1}\log \left[x\log^2\frac{e}{x}\right]^{-1}\text dx = -\infty.
\]
%
\end{example}
\begin{example}
For $m\geq 2,$ the similar properties has the density
\begin{equation*}
f(x) =c_2(m)\left[\|x\|^m \log^2\frac{e}{\|x\|}\right]^{-1}\mathbbm{1}_{B_1(0)}(x), \quad x \in \Re^{m},
\end{equation*}
where
$c_2(m) = \Gamma(\frac{m}{2})/(2\pi^{m/2}).$ 
That is $f$ has finite moments but $H(f)=-\infty.$
\end{example}

If a random vector $X$ in $\Re^{m}$ has a bounded density $f$ with $\|f\|_\infty=\sup_{x\in \Re^m}f(x)<\infty$, then there is a lower bound for its entropy \cite{bobkov2011})
\begin{equation}
\label{eq:log_f}
    \frac{1}{m}H(f)\geq \log\|f\|_\infty^{-1/m}.
\end{equation}

If, in addition, $f$ is log-concave (that is, $\log f$ is concave),  then
\[
  \log\|f\|_\infty^{-1/m} \leq  \frac{1}{m}H(f) \leq 1 + \log \|f\|_\infty^{-1/m}.
\]
Moreover, provided the existence of $p$th moment $\|X\|_p = \left\{\expe\|X\|^p\right\}^{1/p}<\infty$, $p\geq 1,$ one has for a log-concave density $f,$ see \cite{marsiglietti2018},
\begin{equation}\label{eq_H(f)}
   H(f)\geq \log \frac{2\|X - \expe[X]\|_p}{\left[\Gamma(1+p)\right]^{1/p}}.
\end{equation}
If $m=1$, then for symmetric log-concave random variable
\begin{equation}\label{eq_|x|_p}
    H(f)\geq \log \frac{2\|X\|_p}{\left[\Gamma(p+1)\right]^{1/p}}, \quad p>-1.
\end{equation}{}
If a symmetric log-concave random vector on $\Re^m$ has finite second moments, then 
\begin{equation}\label{eq_m/2}
   H(f)\geq \frac{m}{2} \log \frac{(\det \Sigma_x)^{1/m}}{c_5(m)}, 
\end{equation}
where $\Sigma_x=\expe\left[(X-\expe X)(X-\expe X)^T\right]$ denotes the  the covariance matrix of $X$ and
\begin{equation}\label{eq_Sigma}
 c_3(m) = \frac{e^2m^2}{4\sqrt{2}(m+2)}.
\end{equation}

Constant $c_3(m)$ can be improved in the case of unconditional random vectors. A function $f:\Re^m\to\Re^m$ is called \emph{unconditional} if for every $(x_1,\ldots,x_m)\in\Re^m$ and $(\varepsilon_1,\ldots,\varepsilon_m)\in \{-1,1\}^{m}$, one has 
\[
f(\varepsilon_1x_1,\ldots,\varepsilon_mx_m ) = f(x_1,\dots,x_m).
\]
For example, the density of standard isotropic Gaussian vector is unconditional. Thus, if 
 $X$ is unconditional, symmetric, and log-concave, then 
\begin{equation}\label{eq_c_5(m)}
    c_3(m) = e^2/2.
\end{equation}{}
The constant ~\eqref{eq_c_5(m)} is better than the constant ~\eqref{eq_Sigma} for $m\geq5$.

\section{Monte-Carlo simulations}
In this section, we illustrate the results of Monte-Carlo simulations by the series of Figures \ref{fig:consistency-kall}--\ref{fig:sw-pvals-k=1}.
\begin{figure*}[!t]
\centering\includegraphics[width=\linewidth]{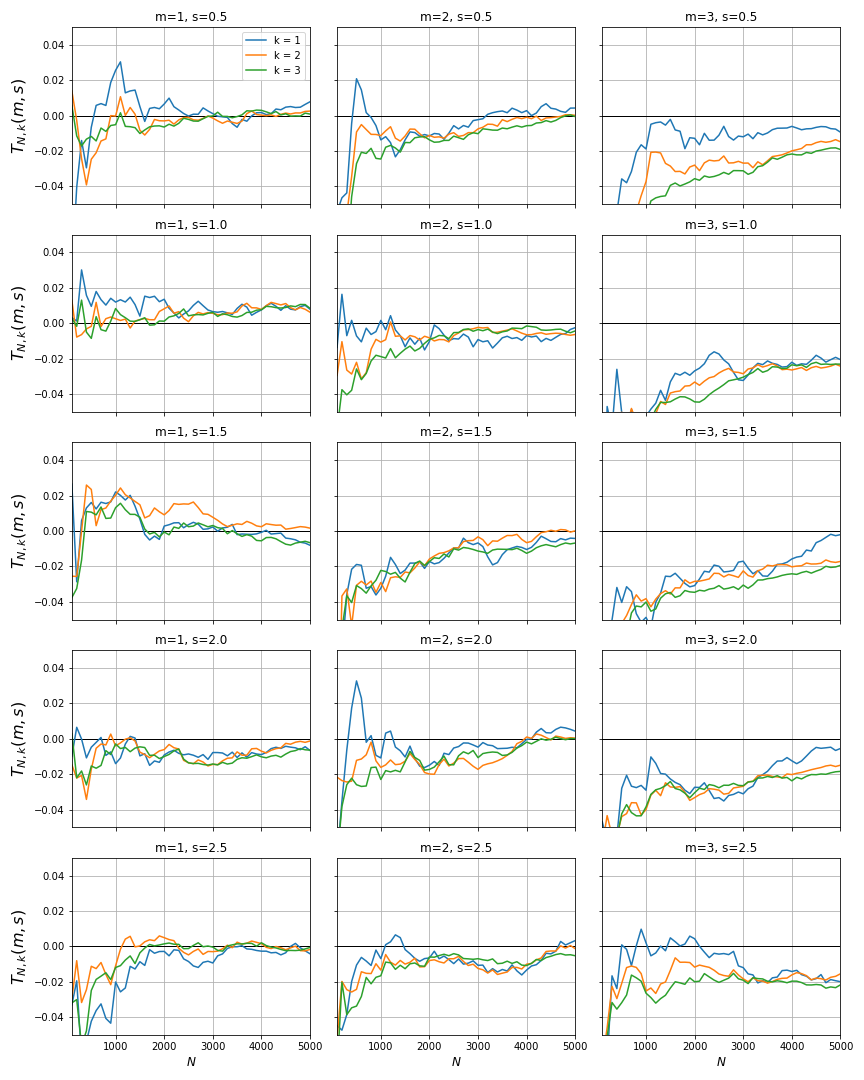}
\caption{Consistency of $T_{N,k}(m,s)$ for different values of $k$ ($M=10$ repetitions).}
\label{fig:consistency-kall}
\end{figure*}
\begin{figure*}[!t]
\centering\includegraphics[width=\linewidth]{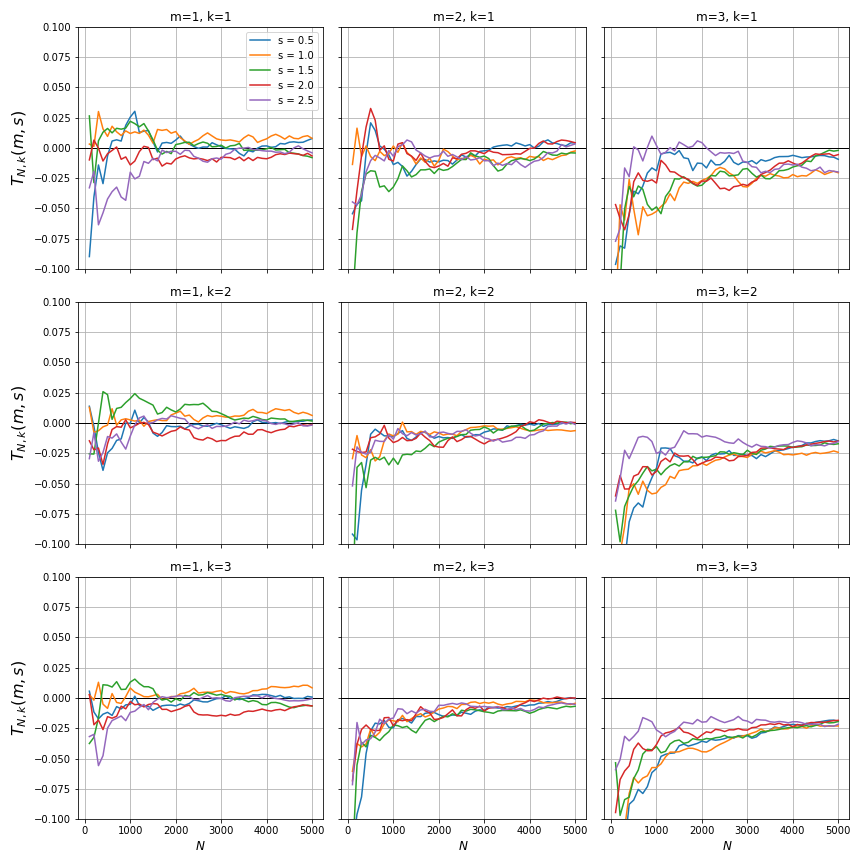}
\caption{Consistency of $T_{N,k}(m,s)$ for different values of $s$ ($M=10$ repetitions).}
\label{fig:consistency-sall}
\end{figure*}
\begin{figure*}[!t]
\centering\includegraphics[height=0.95\textheight,keepaspectratio]{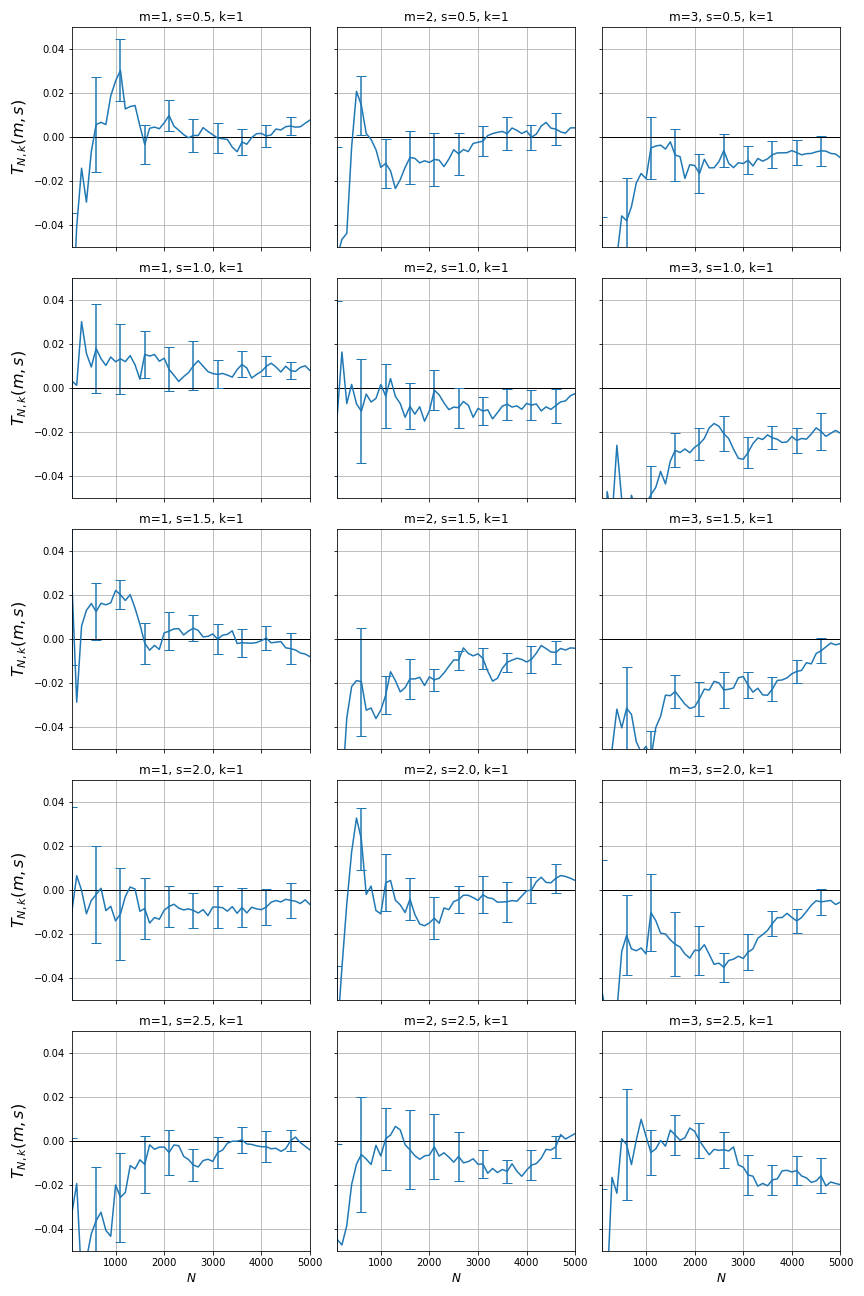}
\caption{Consistency of $T_{N,k}(m,s)$ for $k=1$ ($M=10$ repetitions).}
\label{fig:consistency-k=1}
\end{figure*}
\begin{figure*}[!t]
\centering\includegraphics[width=\linewidth]{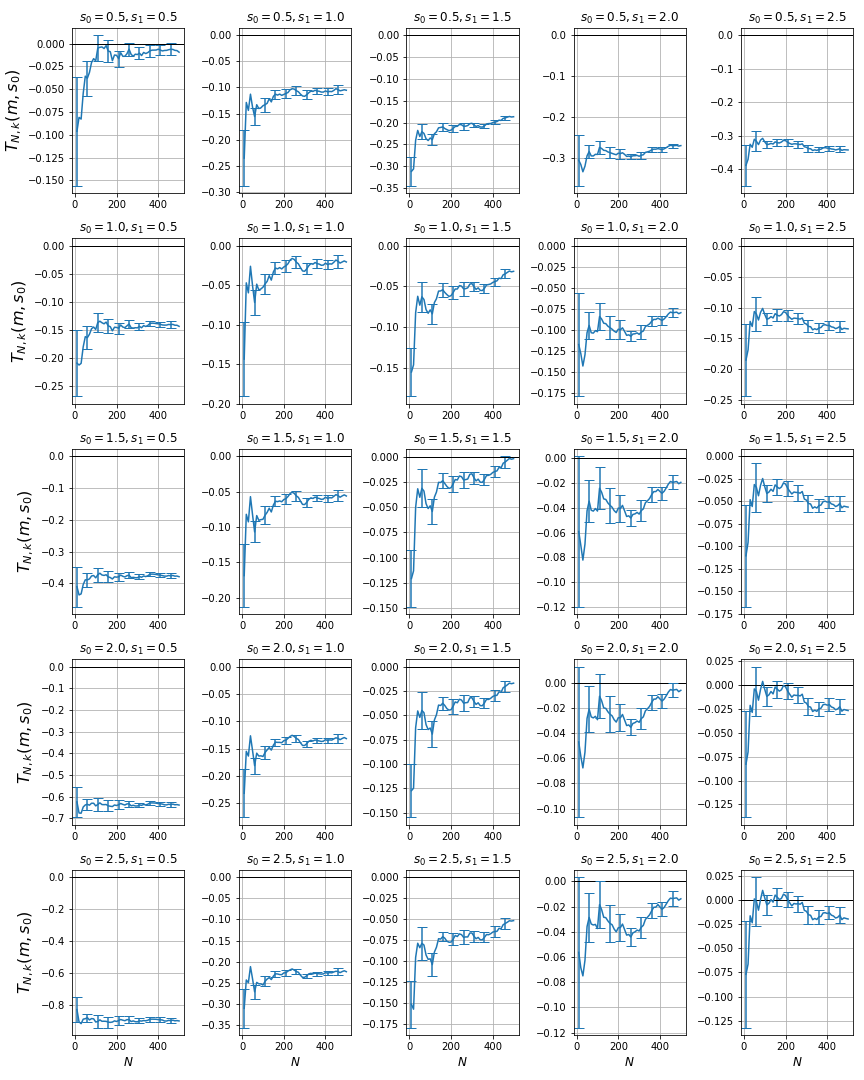}
\caption{The behaviour of $T_{N,k}(m,s_0)$ with $k=1$ on data from the $\GG(m,s_1)$ distribution with $m=2$. }
\label{fig:convergence-GG}
\end{figure*}
\begin{figure*}[!t]
\centering\includegraphics[width=\linewidth]{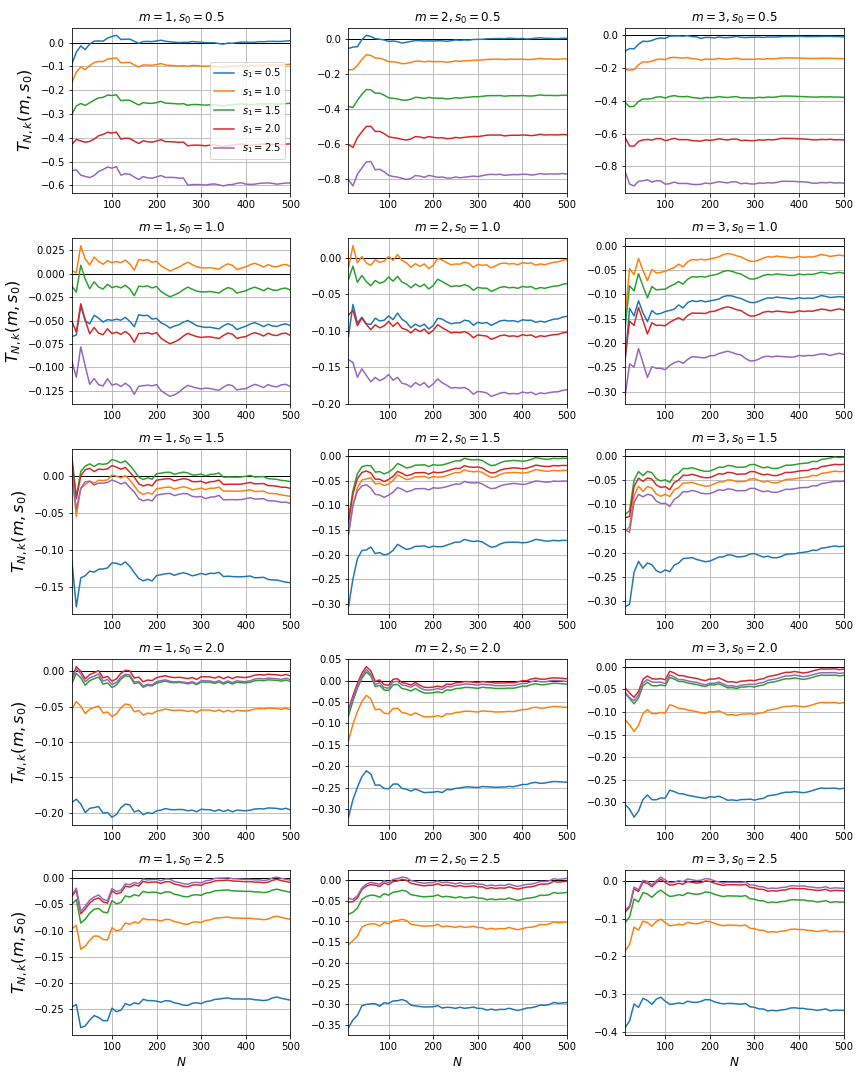}
\caption{The behaviour of $T_{N,k}(m,s_0)$ with $k=1$ on data from the $\GG(m,s_1)$ distribution for different values of $m$  }
\label{fig:convergence-GGall}
\end{figure*}
\begin{figure*}[!t]
\centering\includegraphics[width=\linewidth]{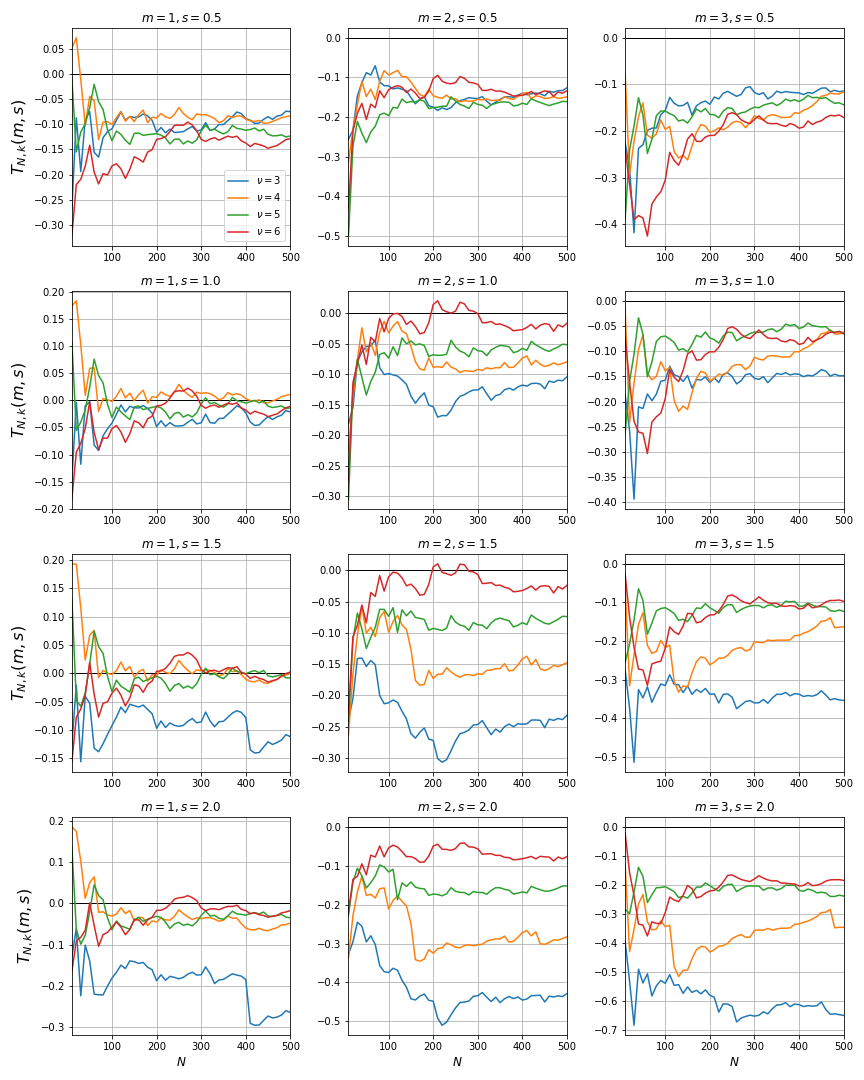}
\caption{The behaviour of $T_{N,k}(m,s_0)$ with $k=1$ on data from the $\ST(m,\nu)$ distribution for different values of $m$.}
\label{fig:convergence-STall}
\end{figure*}
\begin{figure*}[!t]
\centering\includegraphics[width=\linewidth]{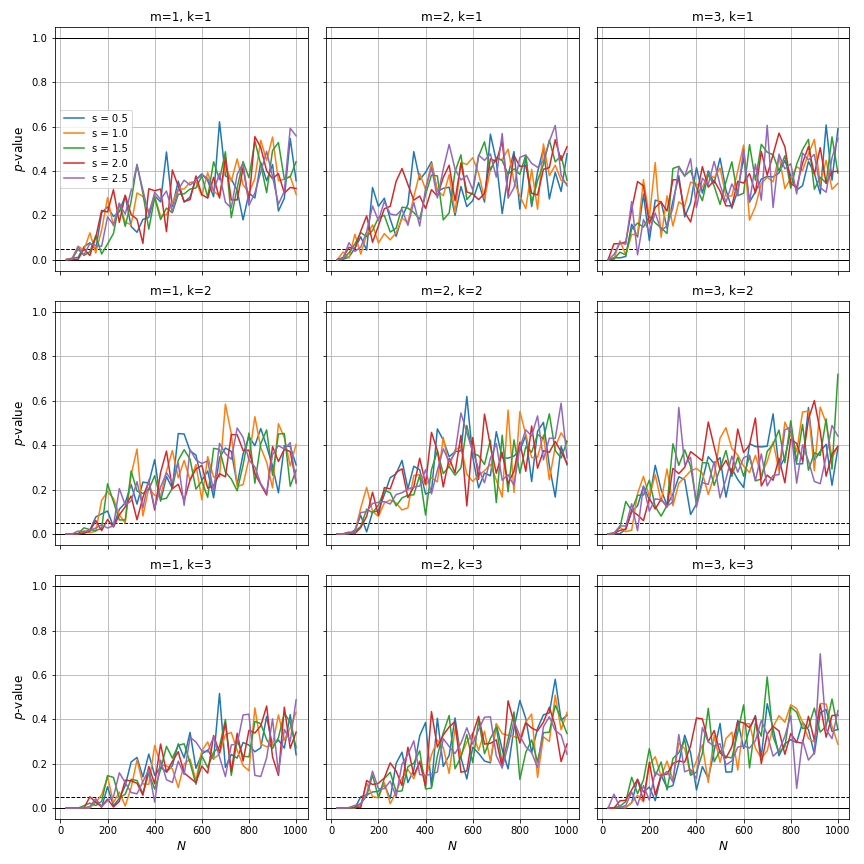}
\caption{Shapiro-Wilk $p$-values as $N$ increases for different values of $m$, $s$ and $k$ ($M=10$ repetitions).}
\label{fig:sw-pvals-all}
\end{figure*}
\begin{figure*}[!t]
\centering\includegraphics[width=\linewidth]{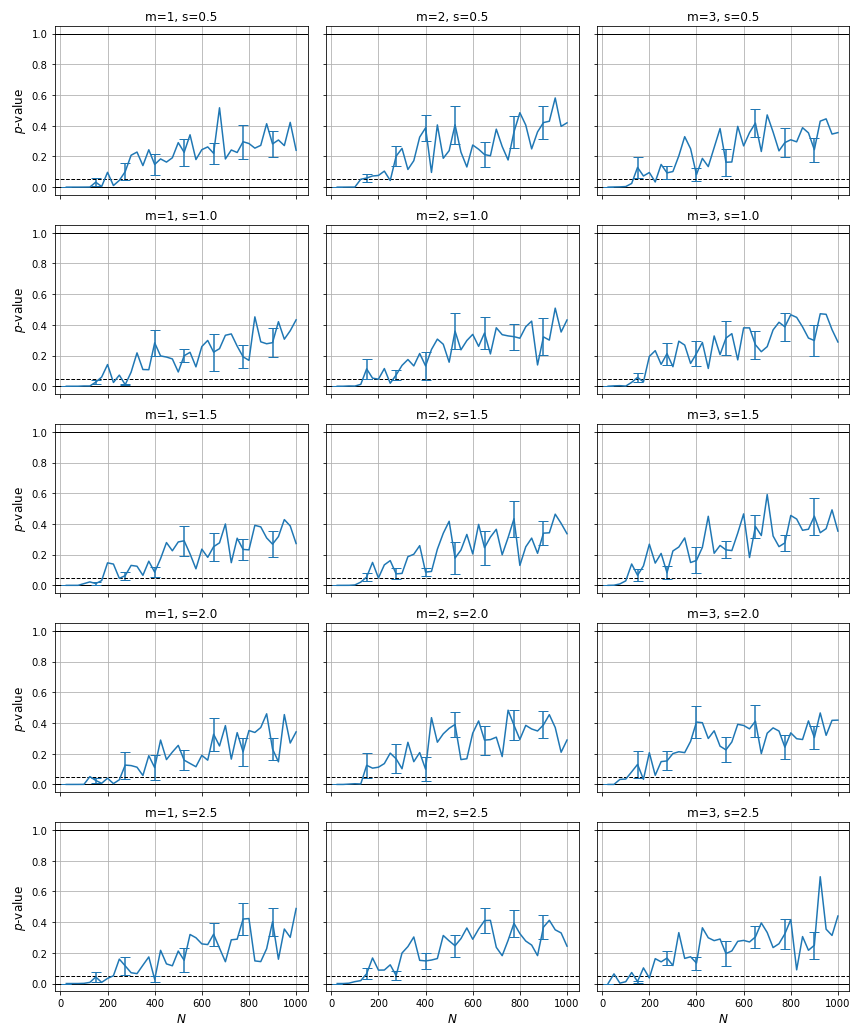}
\caption{Shapiro-Wilk $p$-values as $N$ increases for different values of $m$ and $s$ with $k=1$ ($M=10$ repetitions).}
\label{fig:sw-pvals-k=1}
\end{figure*}

\end{document}